%% file: plq_epssub.tex
\documentclass[11pt]{article}
\usepackage{amsthm,amsmath,amssymb}
\usepackage{algorithm,algpseudocode}
\usepackage{enumerate}
\usepackage{graphicx}
\usepackage{url}
\usepackage[colorlinks=true, linkcolor=blue, urlcolor=blue, citecolor = blue]{hyperref}
\usepackage[margin=1in]{geometry}
\usepackage{caption,subcaption}
\newcommand{\dom}{\mathop{\rm dom}\nolimits}
\newcommand{\gph}{\mathop{\rm gph}}
\newcommand{\plqf}{\mathop{\rm plqf}\nolimits}
\newcommand{\R}{\ensuremath{\mathbb R}}
\newcommand{\Rb}{\ensuremath{\R\cup \{+\infty\}}}

\newcommand{\ubar}[1]{\text{\b{$#1$}}}
\newcommand{\xt}{\widetilde{x}}
\newcommand{\yt}{\widetilde{y}}
\newcommand{\st}{\widetilde{s}}
\newcommand{\xb}{\bar{x}}
\newcommand{\yb}{\bar{y}}
\newcommand{\nan}{\ensuremath{\text{NaN}}}
\newcommand{\qfct}[2]{a_{#1} #2^2 + b_{#1} #2 + c_{#1}}

\input{commands_envs}
\begin{document}
	
	\title{Computation of the Epsilon-Subdifferential of Convex Piecewise-Defined Functions in Optimal Worst-Case Time\thanks{This is a pre-print of an article published in Set-Valued and Variational Analysis (SVAA). The final authenticated version is available online at: \url{http://dx.doi.org/10.1007/s11228-018-0476-5}}}
	\author{Deepak Kumar, Yves Lucet\thanks{Corresponding author; yves.lucet@ubc.ca}}
	\date{\today}
	\maketitle
	
	\begin{abstract}
	The $\epsilon$-subdifferential of convex univariate piecewise linear-quadratic functions can be computed in linear worst-case time complexity as the level-set of a convex function. Using dichotomic search, we show how the computation can be performed in logarithmic worst-case time. Furthermore, a new algorithm to compute the entire graph of the $\epsilon$-subdifferential in linear time is presented. Both algorithms are not limited to convex PLQ functions but are also applicable to any convex piecewise-defined function with little restrictions.
	
	\end{abstract}
	
	\section{Introduction}
	The $\epsilon$-subdifferential quantifies the approximation error intrinsic to numerical computation when one is interested in the solution of a convex nonsmooth optimization problem. It extends Fermat's rule and approximates the convex subdifferential through the Br\o{}ndsted-Rockafellar Theorem~\cite{BRONDSTED-65}. It plays a critical role in the convergence of Bundle methods~\cite{HIRIART-URRUTY-13,HIRIART-URRUTY-93}. 
	
	While the numerical evaluation of an $\epsilon$-subgradient occurs in several numerical optimization algorithms, the full computation of the graph of the subdifferential is too time consuming except for specific classes of functions. Such endeavors are the goal of computational convex analysis, a field that started with the numerical computation of the Legendre-Fenchel transform~\cite{GARDINER-11a,LUCET-96a,LUCET-97b,LUCET-05c,HIRIART-URRUTY-06,LUCET-06,LUCET-12} and has since tackled the computation of the main transforms encountered in convex analysis, e.g. the Moreau envelope, the Lasry-Lions double envelope, the proximal average~\cite{TRIENIS-07, BAUSCHKE-07a,HARE-07,GOEBEL-10,BAUSCHKE-11,JOHNSTONE-11,GOEBEL-12}, etc. See~\cite{LUCET-10} for historical notes and a survey of numerous applications. While pure symbolic computation was considered~\cite{BAUSCHKE-06,BORWEIN-06a,BORWEIN-08}, most work in computational convex analysis use a hybrid symbolic-numerical framework that considers a specific class of functions e.g. piecewise linear or piecewise linear-quadratic functions~\cite{ROCKAFELLAR-09}. Most algorithms have been made publicly available in the CCA open source numerical library~\cite{LUCET-96}. 
	
	More recent work has considered computing operators i.e. multifunctions such as the $\epsilon$-sub\-differential. Bajaj et al.~\cite{BAJAJ-16} proposed an algorithm to compute the $\epsilon$-subdifferen\-tial of a univariate convex PLQ function as the level set of its conjugate minus an affine function, and used the CCA numerical library to carry out that computation in linear time. 
	
	We propose a new algorithm to perform the same computation in logarithmic time by using binary and dichotomic search to avoid computing the entire graph of the conjugate. The approach relies on Goebel's graph-matrix calculus as made explicit in \cite{GARDINER-11a} especially the parametrization of the graph of the conjugate, which has been previously exploited in computational convex analysis in \cite{HIRIART-URRUTY-06}.
	
We then propose another new algorithm to compute the entire graph of the $\epsilon$-subdifferen\-tial in linear time. We use symmetry to simplify the implementation and introduce a new data structure to store $\inf \partial_\epsilon f$, which is a piecewise non-PLQ function for which we give explicit formulas.
	
	Both algorithms are implemented in Scilab within the CCA numerical library~\cite{LUCET-96}. Finally, we point out that both algorithms can be easily adapted to convex piecewise univariate functions that are not necessarily PLQ. 
	
	The remaining of the paper is organized as follow. Section~\ref{s:prel} recalls needed facts and sets the notations, Section~\ref{s:algPt} details the pointwise computation of the $\epsilon$-subdifferential while Section~\ref{s:algGph} explains how to compute its entire graph. We briefly explain in Section~\ref{s:extend} how to extend the algorithms to non-PLQ functions. Finally, Section~\ref{s:conc} concludes the paper and proposes future research directions.
	
\section{Preliminaries and Notations}\label{s:prel}

Throughout this paper, we restrict ourselves to univariate functions and adapt more general definitions and results to that context. Unless otherwise stated, functions are lower semi-continuous (lsc).

The $\epsilon$-subdifferential of a function $f:\R \rightarrow \Rb$ is defined for $x \in \text{dom}(f)= \{x \in \R : f(x) < +\infty \}$ and $\epsilon \geq 0$ as
\[
	\partial_{\epsilon}{f(x)} = \{s \in \R : f(y) \geq f(x) + \langle s,y - x\rangle - \epsilon, \forall y \in \R \}.
\]
and is defined as empty when $x \notin \text{dom}(f)$. Any element $s \in \partial_\epsilon f(x)$ is called an $\epsilon$-subgradient. The function $f$ is said to be proper when it has nonempty domain.

To make the distinction with the approximate subdifferential introduced in~\cite{IOFFE-84}, following \cite{BAJAJ-16} we use the accepted terminology $\epsilon$-subdifferen\-tial, although historically \cite{BRONDSTED-65} used the term ``approximate subgradients''.

When $\epsilon=0$, $\partial_0 f(x)$ reduces to the convex subdifferential that we will denote $\partial f(x)$. In our computation, we will use the fact that $\partial f(x) \subset \partial_\epsilon f(x)$ for any $\epsilon \geq 0$.

The conjugate of $f$ 
\[f^*(s) = \sup_x \{ s x - f(x) \} \]
will play a critical role in the computation due to the following fact.

\begin{fact}[{\cite[Proposition 3.1]{BAJAJ-16}}]
For any function $f:\R \to \Rb$, 
	\[\partial_\epsilon f(\bar{x}) = \{s \in \R : f^*(s) \leq l_{\bar{x}}(s)\}\]
where $l_{\bar{x}}: s \mapsto \epsilon-f(\bar{x}) + \langle s,\bar{x} \rangle$.
\end{fact}
	
While \cite{BAJAJ-16} uses the CCA numerical library~\cite{LUCET-96} to compute $f^*$ explicitly, we will avoid such computation by relying on the natural parametrization of $f^*$ previously exploited in ~\cite{GARDINER-11a,HIRIART-URRUTY-06}.

\begin{fact}\label{f:equiv}
For a lsc convex function $f:\R\to\Rb$, the following are equivalent
\begin{enumerate}[(i)]
	\item $(x,s,y) 	\text{ satisfies } y=f(x), 		s\in \partial f(x)$,
	\item $(s,x,y^*)\text{ satisfies } y^*=f^*(s),x\in \partial f^*(s), y^*=sx-y$.
\end{enumerate}
\end{fact}

A function $f:\R \to \Rb$ is \emph{piecewise linear-quadratic} (PLQ) if $\dom(f)$ can be represented as the union of finitely many closed intervals on each of which $f$ is linear or quadratic. Note that a PLQ function is continuous on its domain and can be represented in the form
   	\begin{equation}\label{eq:PLQ}
   	f(x) = \begin{cases}
		a_{0}x^{2} + b_{0}x + c_{0}, &\text{ if }  -\infty < x < x_{0} \\
		a_{1}x^{2} + b_{1}x + c_{1}, &\text{ if }  x_{0} \leq x \leq x_{1}\\
		\vdots & \vdots \\
		a_n x^{2} + b_n x + c_n, &\text{ if } \hspace{5pt}  x_{n-1} < x < +\infty,
		\end{cases}
   	\end{equation}
where $a_i, b_i, c_j \in \R$ for $i= \{0,1,\cdots,n\}$, $j= \{1,\cdots,n-1\}$, and $c_0, c_n \in \Rb$. (The above formula represents more general functions than PLQ functions since some coefficient choices may result in a function that is discontinuous on its domain.) 

Any PLQ function will be stored as a $4\times (n+1)$ matrix~\cite{LUCET-06}
\begin{equation}
   P = \begin{bmatrix}
   x_{0} & a_{0} & b_{0} & c_{0} \\
   x_{1} & a_{1} & b_{1} & c_{1} \\
   \vdots& \vdots & \vdots & \vdots\\
   x_{n-1} & a_{n-1} & b_{n-1} & c_{n-1}\\
   +\infty & a_{n} & b_{n} & c_{n}
   \end{bmatrix},
\label{eq:P}
\end{equation}
with the convention that if $c_{0} = +\infty$ or $c_{n} = +\infty$, then the structure demands that $a_{0} = b_{0} =0$ or $a_{n} = b_{n} =0$ respectively. A quadratic function on $\R$ is stored as $n=0$ and $x_{0} = +\infty$ while an indicator function of a single point $\tilde{x} \in \R$
\[
   f(x) = \iota_{\{\tilde{x}\}}(x) + c = 
		\begin{cases}
		c,  & \text{ if } x = \tilde{x}\\
   +\infty, & \text{ if } x \neq \tilde{x}
		\end{cases}
\]
   where $ c \in \mathbb{R}$, is stored as a single row vector 
	$ P= \begin{bmatrix}
   \tilde{x} & 0 & 0 & c
   \end{bmatrix}$.
We will call the later function, a needle function.
	
\section{Computing \texorpdfstring{$\partial_\epsilon f(\bar{x})$}{epssub}}\label{s:algPt}
		 In this section, we present Algorithm~\ref{algPt} to compute the $\epsilon$-subdifferential of an univariate proper convex PLQ function at a particular $\bar{x} \in \dom(f)$. The key idea is to use dichotomic search to reduce the time complexity. 
		
		For a PLQ function $f:\R \rightarrow \Rb$ defined as \eqref{eq:PLQ} we note $s_i=2 a_i x_i + b_i$ and $\partial_\epsilon f(\bar{x}) = [\ubar{s},\bar{s}]$ where $\ubar{s}=\inf \{s : f^*(s) \leq l(s) \}$ and $\bar{s}=\sup \{s : f^*(s) \leq l_{\bar{x}}(s) \}$. We also note $\dom(f) = [\ubar{d}, \bar{d}]$ i.e. $\ubar{d}=\inf \{x: f(x)<+\infty\}$.

		Since $f^*$ is proper convex PLQ the equation $f^*(s)=l(s)$ has at most two solutions. We focus on the computation of $\ubar{s}$ (the computation of $\bar{s}$ is similar). 
		
		If $\bar{x}=\ubar{d}$, $f^*(s)<l_{\bar{x}}(s)$ for all $s < \inf \partial f(\bar{x})$. Hence, $\ubar{s}=-\infty$.
		
		So let us assume there is a single solution to $f^*(s)=l_{\bar{x}}(s)$ on the interval $(-\infty, \inf \partial f(\bar{x})]$.		
		The first step of Algorithm~\ref{algPt} is to locate the interval $[s_l,s_{l+1}]$ containing $\ubar{s}$. When $\ubar{s}\leq s_1$, we can compute $\ubar{s}$ directly so we now assume  $s_1 \leq \ubar{s}$. Using binary search we locate the index $i$ such that $x_{i-1} < \bar{x} \leq x_i$. We know that $\ubar{s} \leq s_i$ since $s_i \in \partial f(\bar{x}) \subset \partial_\epsilon f(\bar{x})$ i.e. we have $s_1 \leq \ubar{s} \leq s_i$, see Figure~\ref{f:conj_l}.
		
		\begin{figure}%
			\centering
			\def\svgwidth{0.6\textwidth}
			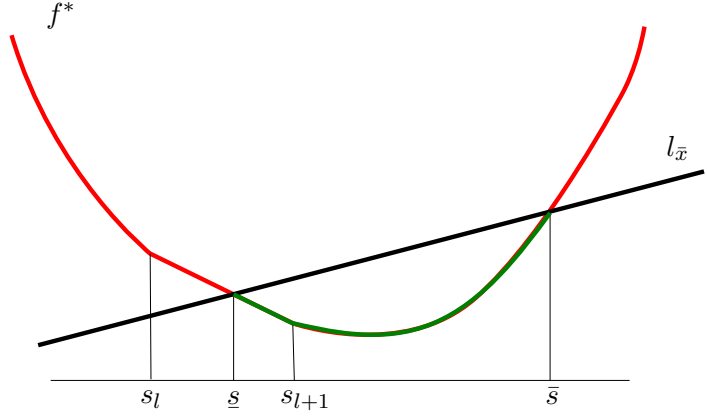		
		\caption{Conjugate $f^*$ (in red) with the affine function $l_{\bar{x}}$ (in black). The algorithm computes the index $l$ such that $s_l \leq \ubar{s} \leq s_{l+1}$. The green area corresponds to $\partial_\epsilon f(\bar{x})=[\ubar{s},\bar{s}]$.}%
		\label{f:conj_l}%
		\end{figure}
		
		We then perform a dichotomic search between $l=1$ and $u=i$ by computing the middle index $m=\lfloor (l+u)/2 \rfloor$ and updating $l$ or $u$ depending on whether $f^*(s_m)<l(s_m)$. Using Fact~\ref{f:equiv}, we have $y_m=f(x_m), s_m=2 a_m x_m + b_m\in \partial f(x_m)$ so $f^*(s_m)=y^*_m = s_m x_m - y_m$ i.e. we can perform the dichotomic search without computing $f^*$ explicitly, see Algorithm~\ref{algPt}.
	
\begin{algorithm}
\caption{Computing the $\epsilon$-subdifferential in logarithmic time}
\label{algPt}
\begin{algorithmic}[1]
	\Require $\plqf$ in PLQ matrix format, $\bar{x} \in \dom(f)$, $\epsilon >0$
	\Statex $x_i=\plqf(i,1)$
	\Function{plq\_epssub}{$\plqf, \bar{x}, \epsilon$}
		\Statex Dichotomic search on left part
		\State Find $i$ such that $x_{i-1} < \bar{x} \leq x_i$ using binary search
		\State $l=1$; $u=i$
		\While{$u-l > 1$}
			\State $m=\lfloor (l+u)/2 \rfloor$; $a = \plqf(m,2)$; $b=\plqf(m,3)$; $ c=\plqf(m,4)$
			\State $s=2 a x_m + b$; $y=a x_m^2 + b x_m + c$; $y^*=s x_m - y$
			\If{$y^* > l(s)$} 
				\State 	$l = m$ 
			\Else 	
				\State $u = m$
			\EndIf
		\EndWhile
		\Statex The lower bound to $\partial_\epsilon f(\bar{x})$ is in $[s_l,s_{l+1}]$
		\State $\ubar{s}$ = Intersection(l,u); \Comment{return $\inf \{s : f^*(s) < l(s) \}$}
		\Statex Perform similar dichotomic search on right part to obtain $\bar{s}$
	 \State \textbf{return} $\ubar{s},\bar{s}$
\EndFunction
\end{algorithmic}
\end{algorithm}
		 
Invoking the function \emph{Intersection(l,u)} computes $\ubar{s}=\inf \{s : f^*(s) < l(s) \}$ by considering several cases. The dichotomic search ensures that $\ubar{s}\in [s_l,s_{l+1}]$, and that the equation $f^*(s)=l(s)$ has at most a solution in that interval. If $f^*(s)=l(s)$ has a solution in $[s_l,s_{l+1}]$, we compute explicitly $f^*$ on that interval by interpolation ($f^*$ is at most quadratic and goes through $s_j,s_j x_j - y_j$ with derivative $2 a_j s_j + b_j$ for $j=l, l+1$) and solve the resulting linear or quadratic equation. Otherwise we have $f^*(\ubar{s}) < l(\ubar{s})$, and in that case, $\ubar{s} = s_l$.
	
	The algorithm computes $\bar{s}$ similarly.
		
		\begin{proposition}
			Given an univariate convex PLQ function $f$ represented as \eqref{eq:P}, $\bar{x} \in \dom(f)$, and $\epsilon >0$, Algorithm~\ref{algPt} returns $\partial f(\bar{x})=[\ubar{s},\bar{s}]$ with $\ubar{s},\bar{s}\in \R\cup \{-\infty, +\infty\}$ in $O(\log n)$ time.
		\end{proposition}
		
		\begin{proof}
		The algorithm performs the same computation as \cite[Algorithm 2]{BAJAJ-16}, which proves its correctness. Its complexity is logarithmic since it performs constant time operations (solving a quadratic equation), a binary search, and up to 2 dichotomic searches in sequential order.
		\end{proof}
		
To validate the complexity numerically, we build a convex plq function with a large number of pieces by sampling the function $f(x)=x^4$, then building a (zeroth-order) piecewise linear approximation and finally building its Moreau envelope. The resulting function is convex PLQ and is alternating between being linear and quadratic. We then time the original algorithm and our new algorithm. When the function contains around $40,000$ pieces, the original linear-time algorithm takes $16.4$ seconds while our algorithm takes $0.03$ seconds on a desktop quad-core hyperthreading Intel Xeon with 64GB of memory running Windows 7 64 bits and Scilab 5.5.2 (64 bits). Table~\ref{t:timing1} shows the resulting timings.  

% Table generated by Excel2LaTeX from sheet 'Yves_Lamport'
\begin{table}[htbp]
  \centering
  \caption{Computation time comparison between the linear time original algorithm from \cite{BAJAJ-16} (column $O(n)$) and our new logarithmic time algorithm (column $O(\log n)$; $n$ is the number of pieces of the plq function used. Computation times are in seconds.}
    \begin{tabular}{rrr}
		\hline
    \multicolumn{1}{l}{n} & \multicolumn{1}{l}{$O(n)$} & \multicolumn{1}{l}{$O(\log n$)} \\
		\hline
    4,000  & 1.0   & 0.02 \\
    8,000  & 2.3   & 0.02 \\
    12,000 & 3.8   & 0.03 \\
    16,000 & 5.3   & 0.03 \\
    20,000 & 7.3   & 0.03 \\
    24,000 & 9.2   & 0.03 \\
    28,000 & 11.4  & 0.03 \\
    32,000 & 12.9  & 0.02 \\
    38,000 & 14.5  & 0.02 \\
    40,000 & 16.4  & 0.03 \\
		\hline
    \end{tabular}%
  \label{t:timing1}%
\end{table}%

%\begin{figure}%
%\centering
%\includegraphics[trim=1cm 8cm 1cm 8cm, clip, width=\textwidth]{timings2}%
%\caption{Computation time of \cite{BAJAJ-16} algorithm that runs in $O(n)$ vs. Algorithm~\ref{algPt} that runs in $O(\log n)$.}%
%\label{f:timing}%
%\end{figure} 

\section{Computing \texorpdfstring{$\gph \partial_\epsilon f$}{algGph}}	\label{s:algGph}
The graph of $\partial_\epsilon f$ is defined as 
\[
\gph \partial_\epsilon f = \{ (x,s) : s\in \partial_\epsilon f(x) \}.
\]
Our objective in this section is to fully describe $\gph \partial_\epsilon f$ for bivariate convex PLQ functions. 

First, we note that we only need to describe the lower and upper bound of $\gph \partial_\epsilon f$. Define $\ubar{g}, \bar{g}: \R\to \R\cup\{-\infty, +\infty\}$ by $\ubar{g}(x) = \inf \partial_\epsilon f(x) = \inf\{ s : s\in \partial_\epsilon f(x)\}$, and $\bar{g}(x) = \sup \partial_\epsilon f(x)$. Then $\gph \partial_\epsilon f = \{ (x,s) : \ubar{g}(x) \leq s \leq \bar{g}(x) \}$.
	
Next, only $\ubar{g}$ needs to be computed as the following lemma, which follows directly from the definition of $\partial_\epsilon f(x)$, indicates.

\begin{lemma}\label{l:ub}
Assume $\epsilon>0$, $f$ is a proper function, and denote $h(x)=f(-x)$. Then 
\[
s \in \partial_\epsilon f(x) \Leftrightarrow -s \in \partial_\epsilon h(-x)
\] 
and in particular
\[
\sup \partial_\epsilon f(x) = - \inf \partial_\epsilon h (-x).
\]
\end{lemma}
	
So we only need to compute $\ubar{g}$; then we apply the same algorithm to $h(x)=f(-x)$ to deduce the upper bound.

Before explaining the algorithm, we define our notations. Given $\epsilon>0$ and a convex lsc PLQ function $f$, we note $\xb$ the point at which we wish to evaluate $\ubar{g}(\xb)$. We name $\cal L$ the line going through $(\xb, \yb-\epsilon)$ with $\yb=f(\bar{x})$ and tangent to the graph of $f$ at a point $(\xt, \yt)$ with $\yt=f(\xt)$ and $\xt < \bar{x}$. The line $\cal L$ has equation $y = (x - \xt) + \yt$ and is illustrated on Figure~\ref{f:t1} when $f$ is smooth at $\xt$ and on Figure~\ref{f:t2} when $f$ is nonsmooth at $\xt$.

\begin{figure}%
\centering
\includegraphics[width=.7\columnwidth]{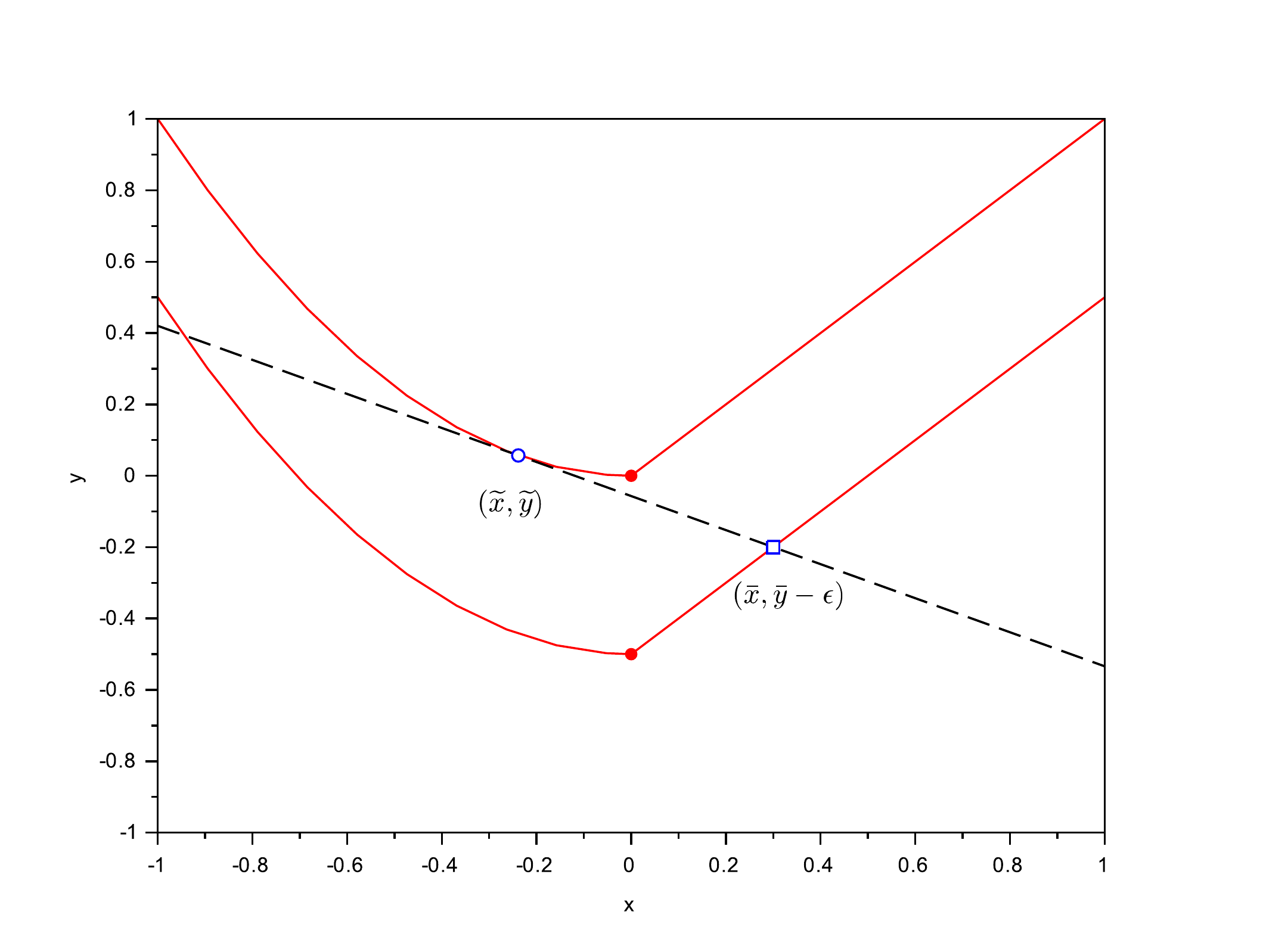}%
\caption{Case 1: $f$ smooth at $(\xt, \yt)$. The dashed line is $\cal L$.}%
\label{f:t1}%
\end{figure}

\begin{figure}%
\centering
\includegraphics[width=.7\columnwidth]{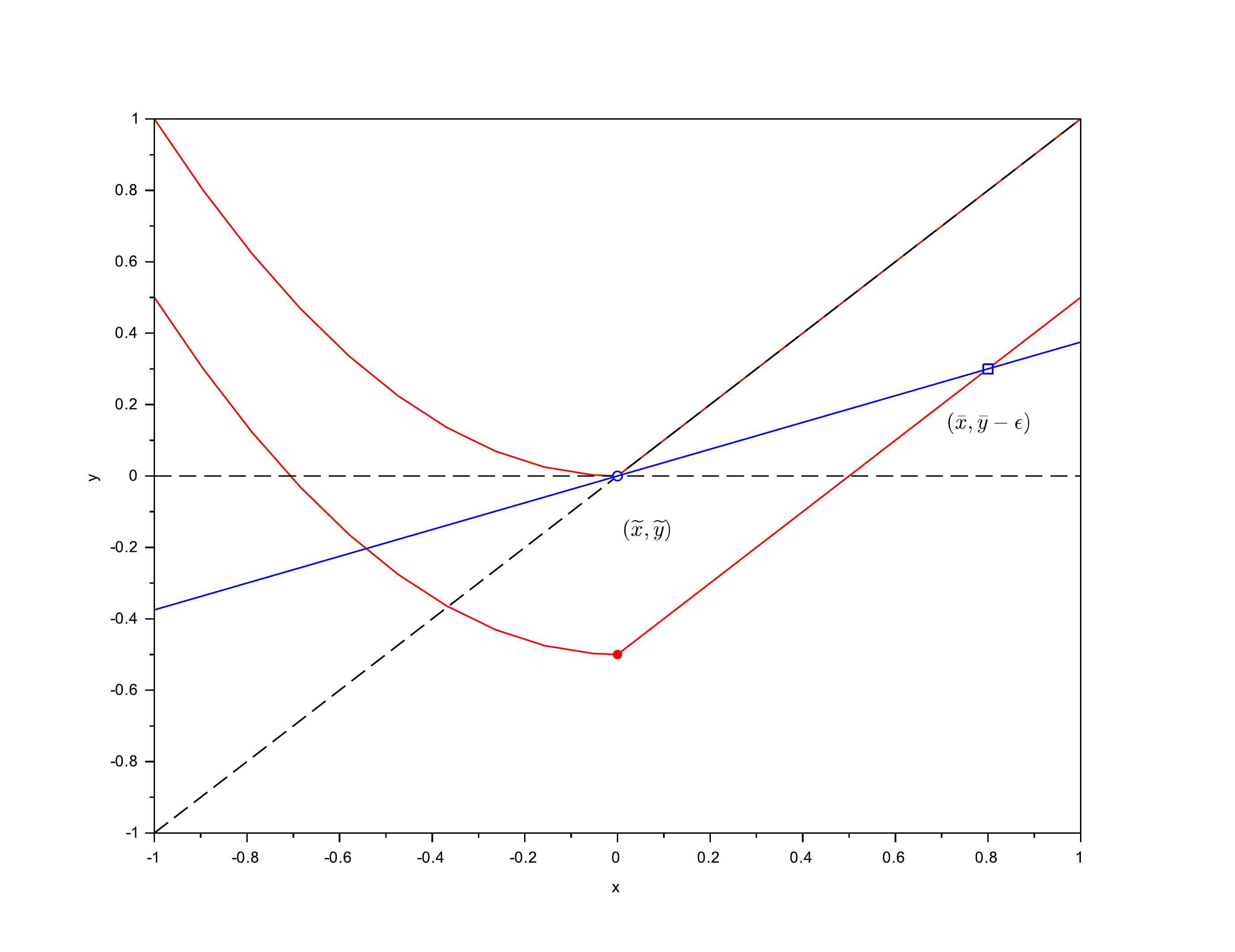}%
\caption{Case 2: $f$ nonsmooth at $(\xt, \yt)$. The dashed lines visualize the subdifferential $\partial f(\xt)$ while the blue line is $\cal L$.}%
\label{f:t2}%
\end{figure}

The function $\ubar{g}$ is a piecewise function but not a PLQ function. In order to make our algorithm more generic, we adopt a data structure that only stores enough information to evaluate $\ubar{g}$ with the knowledge of the PLQ matrix $P$ of $f$. We will store the output using an $m\times 5$ matrix $L$ where each row stores a piece of $\ubar{g}$ in the format $[x, t, \tilde{i}, \bar{i}, v]$. Similarly to a PLQ matrix, the first column stores a sorted array of points. The second column stores a type index $t$ where $t=1$ if the function $f$ is smooth at the point $\xt$ associated with $\xb$ (see Figure~\ref{f:t1}) while $t=2$ if $f$ is nonsmooth at $\xt$ (see Figure~\ref{f:t2}). The value $t=3$ indicates that $\partial_\epsilon f$ is constant on that interval with value $v$ (see Figure~\ref{f:t3}). The index $\tilde{i}$ (resp. $\bar{i}$) refers to the index in the PLQ matrix of the piece storing $\xt$ (resp. $\xb$). 

\begin{figure}%
\centering
\includegraphics[width=.7\columnwidth]{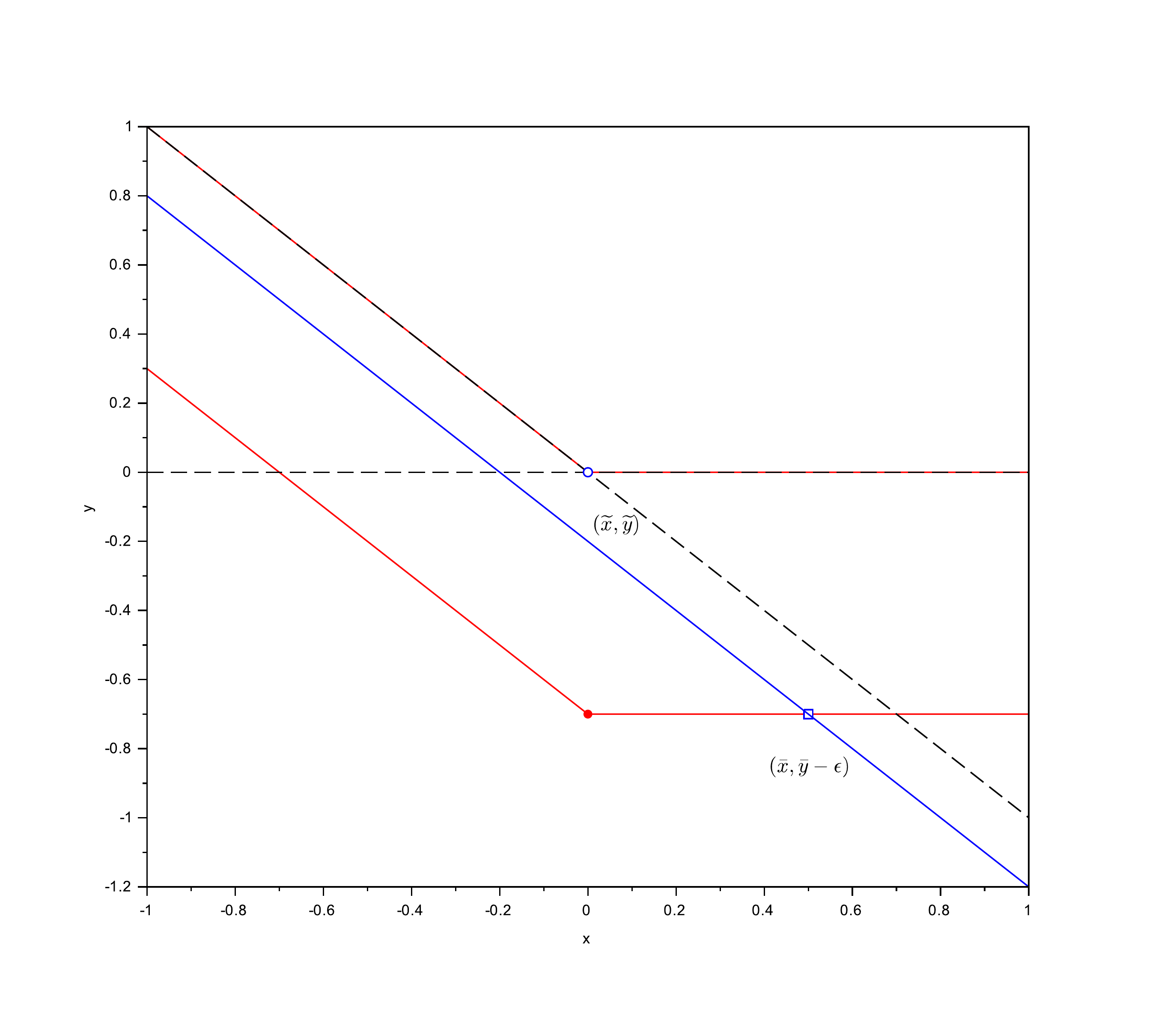}%
\caption{Case 3: $\inf \partial_\epsilon f$ constant around $(\xb, \yb)$.}%
\label{f:t3}%
\end{figure}

Once computed, the matrix $L$ allows us to obtain $\ubar{g}(\xb)$ as follow. Giving $\xb$, we find the row $k$ in $L$ such that $L(1,k-1) < \xb \leq L(k,i)$. We then obtain the type $t_k$. If $t_k=3$, we immediately return $\ubar{g}(\xb)=v_k$. If $t_k=2$, we return $s$, the slope of the line $\cal L$ going through $(\xb, \yb-\epsilon)$ and $(\xt, \yt)$ where $\xt=P(\tilde{i}_k,1)$ and $\yt=f(\xt)$. Otherwise, $t_k=1$ and we return the slope $s$ of the line $\cal L$ going through $(\xb, \yb-\epsilon)$ and tangent to the quadratic function $q_k$, which corresponds to the $\tilde{i}_k$ piece of $P$. The later requires solving a quadratic equation. Our argument proved the following result.

\begin{proposition}
Given a PLQ function $f$ with $n+1$ pieces stored as PLQ matrix $P$, $\epsilon>0$, the matrix $L$ storing the function $\ubar{g} = \inf \partial_\epsilon f$, and a point $\xb$; the value $\ubar{g}(\xb)$ can be computed in logarithmic time. 

In addition, given a PLQ matrix corresponding to $h(x)=f(-x)$, and a matrix $U$ storing $\bar{g}=\sup \partial_\epsilon f$; the value $\bar{g}(\xb)$ can be computed in logarithmic time. Consequently, the set $\partial_\epsilon f(\xb)$ can be computed in logarithmic time.

Given a sorted set $X_b$ of $m$ values for $\xb$ and all of the above input, computing $\partial_\epsilon f(\xb)$ for $\xb \in X_b$ can be performed in $O(\min(m \log n, n + m))$.
\end{proposition}

\begin{proof}
First, given a single value $\xb$, we search for the row in $L$ such that $L(k-1,1) < \xb \leq L(k,1)$ using binary search. Then a constant time calculation gives $\ubar{g}(\xb)$ (obvious for $t_k=3$ and $t_k=2$; by Lemma~\ref{l:cst} below for $t_k=1$). We perform the same computation for $U$ and $\bar{g}$ to obtain $\partial_\epsilon f(\xb)$ in logarithmic time.

Then given a sorted set $X_b$ of $m$ values, we can repeat the same computation in $O(m \log n)$ worst-case time. Alternatively, when $m$ is large, we can do a linear search of $L(:,1)$ and consider each point in $X_b$ or $L(:,1)$ only once, resulting in a $O(n+m)$ evaluation algorithm.
\end{proof}

The previous result relies on solving the following equations in constant time. The algorithm performs that computation in the subroutine esub\_compute\_xb whose code is not included since it only includes special cases checks and solving a quadratic (the full algorithm is available for download in the CCA toolbox or by contacting the corresponding author).

\begin{lemma}\label{l:cst}
Defined $\xt < \xb$ as the point on the graph of $f$ at which the line going through $(\xb, f(\xb)-\epsilon)$ is tangent to $\gph f$ with $\epsilon>0$. Assume $f$ is differentiable at $\xt$. Note $\tilde{i}$ (resp. $\bar{i}$) the row index in the PLQ matrix of $f$ for the piece containing $\xt$ (resp. $\xb$), i.e. $P(\tilde{i}-1,1) < \xt \leq P(\tilde{i},1)$ (resp. $P(\bar{i}-1,1) < \xb \leq P(\bar{i},1)$). Given $\xt$, $\tilde{i}$, and $\bar{i}$, the point $\xb$ is a solution of 
\[
\begin{cases}
\yb - \epsilon & =  \st (\xb - \xt) + \yt,\\
\st & =  p'_{\tilde{i}} ( \xt), \\
\yb & =  p_{\bar{i}}(\xb),\\
\yt & =  p_{\tilde{i}}(\xt);
\end{cases}
\]
where $p_k$ is the quadratic function corresponding to the $k$\textsuperscript{th} piece of $f$, and $p'_k$ is the derivative of $p_k$. Conversely, given $\xb$, the point $\xt$ can be computed as the solution to the same equations.
\end{lemma}

\begin{proof}
The equations translate the facts that $(\xb, \yb-\epsilon)$ is on the line going through $(\xt, \yt)$ with slope $\st$ with $\st$ the derivative of $f$ at $\xt$, and $\yb$ (resp. $\yt$) the image of $\xb$ (resp. $\xt$) by $f$.

Substituting the variables and the coefficient of the quadratic functions $p_i(x)=\qfct{i}{x}$, we obtain the quadratic equation 
\begin{equation}
\qfct{\bar{i}}{\xb} - \epsilon = (2 a_{\tilde{i}} \xt + b_{\tilde{i}}) (\xb - \xt) + \qfct{\tilde{i}}{\xt},
\label{e:q}
\end{equation}
which can be solved explicitly in constant time. While that quadratic equation may have 0, 1, or 2 solutions, the assumptions always ensure there is at least one solution, and the geometric positions of $\xt$ with respect to $\xb$ always allow us to pick the correct root.
\end{proof}

\begin{remark}
It is always possible to adopt an explicit data structure similar to the PLQ matrix by expliciting the formula for $\st$ as a function of $\xb$. The resulting data structure only works with PLQ functions while the one suggested will be extended beyond PLQ functions in the following section.

More precisely, $\ubar{g}$ is a piecewise function. Around a point $\xb$, the type $t_k$ indicates the explicit formula to use. If $t_k=3$, $\ubar{g}(\xb)=v$, i.e. $\ubar{g}$ is constant around $\xb$. If $t_k=2$, $\ubar{g}(\xb)=(\yb-\yt)/(\xb-\xt)$ or more explicitly,
\[
\ubar{g}(x)=\frac{\qfct{\bar{i}}{x}-\yt}{x-\xt}
\]
where $\xt$ is constant on that interval. Finally, if $t_k=3$, $\ubar{g}(\xb)$ is the solution of \eqref{e:q} where $\xt$ is now the moving variable, i.e. on that interval
\[
\ubar{g}(x) = -2 a_{\tilde{i}} x \pm \sqrt{2(a_{\tilde{x}} x)^2 + a_{\tilde{i}}
\left(
	\qfct{\bar{i}}{x} - b_{\tilde{i}} x + c_{\tilde{i}} -\epsilon
\right)
}.
\]
In conclusion, the function $\ubar{g}$ is piecewise and on each piece it is either a constant, a rational function, or a linear plus the square root of a quadratic function.
\end{remark}

\begin{example}
For example, the absolute value function has PLQ matrix 
\[
P =\begin{bmatrix}
0 & 0 & -1 & 0\\
\infty & 0 & 1 & 0
\end{bmatrix}
\]
and its associated function $\ubar{g}=\inf \partial_\epsilon f$ is stored as
\[
L = \begin{bmatrix}
0.25 & 3 & \nan & \nan & -1\\
\infty & 2 & 1 & 2 & \nan
\end{bmatrix}.
\]
The $\epsilon$-subdifferential of the absolute value is illustrated on Figure~\ref{f:absprimal} and its full graph is plotted on Figure~\ref{f:absdual}.

The value $\nan$ refers to the standard IEEE 754 Not a Number value. In our context, it indicates that the value should not be used i.e. it is irrelevant. So the first row of $L$ indicates that for any value $\xb \leq 0.25$ $\ubar{g}(\xb)=-1$ (column 2 value of $3$ means it is type $t=3$ i.e. a constant value equal to $v=-1$). The second row of $L$ means that for any $0.25 < \xb < \infty$ the value $\ubar{g}(\xb)$ can be computed as the slope $s$ of the line going through $(\xb, \yb-\epsilon)$ and $(\xt, \yt)$ with $\xt$ (resp. $\xb$) belonging to the piece of $f$ with index $\tilde{i}=1$ (resp ($\bar{i}=2$). In this case, the line goes through $(\xb, \xb-\epsilon)$ and $(0, 0)$ so $s = (\xb - \epsilon)/\xb = 1 - \epsilon / \xb$. 
In the general case, using the information in $L$ with the matrix $P$, we can compute $\ubar{g}(\xb)$. 

Note that since $|-x|=|x|$, $\gph \partial_\epsilon f$ is symmetric with respect to $(0,0)$ as predicted by Lemma~\ref{l:ub}.
\end{example} 

\begin{figure}[t]%
\centering
\includegraphics[width=0.7\columnwidth]{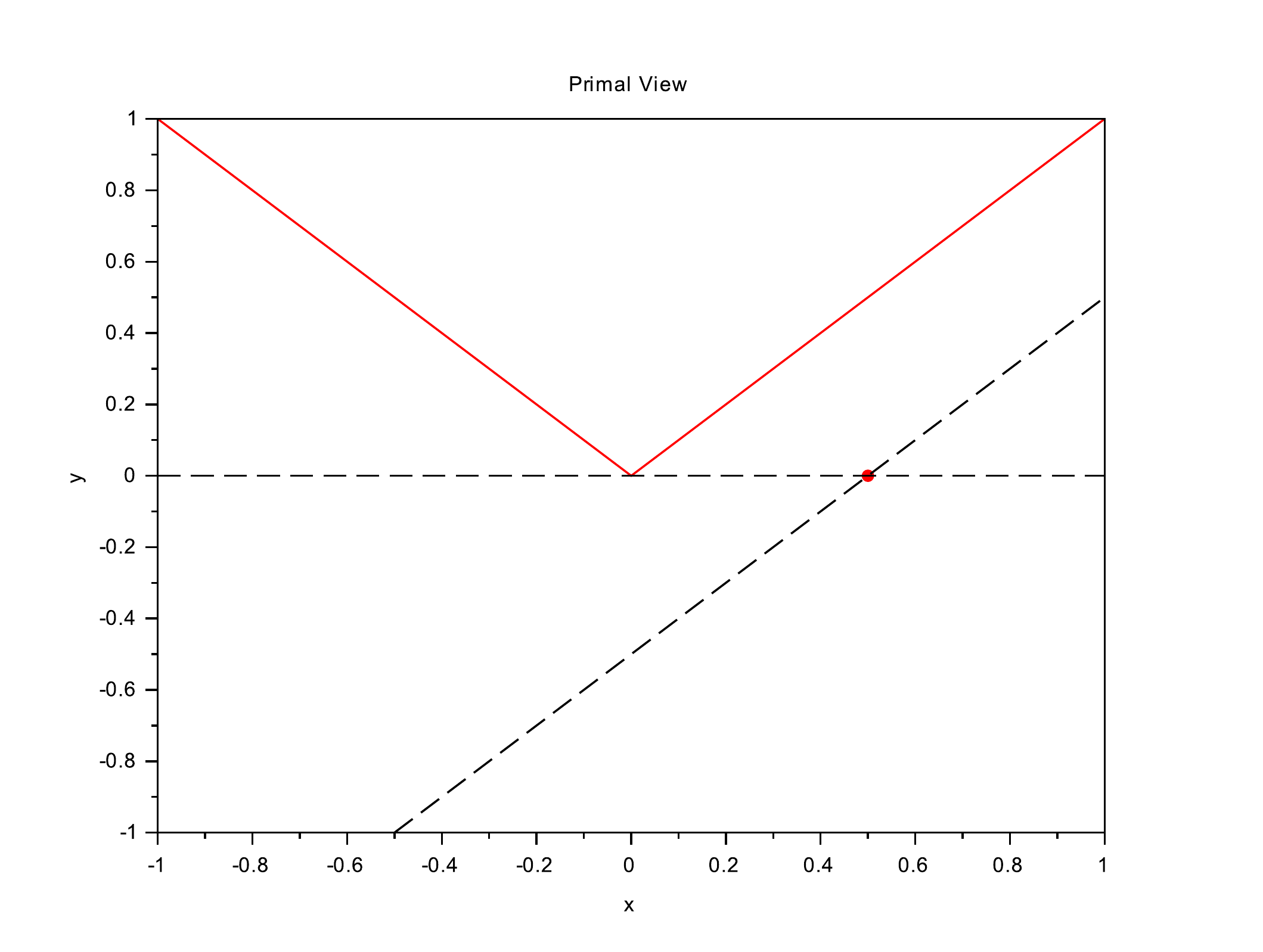}%
\caption{The $\partial_\epsilon f(\xb)$ for $f(x)=|x|$, $\epsilon=0.5$, and $\xb=0.5$. The slope of the horizontal dashed line equals $\ubar{g}(\xb)=\inf \partial_\epsilon f(\xb)$ while the slope of the other dashed line is equal to $\bar{g}(\xb)=\sup \partial_\epsilon f(\xb)$}%
\label{f:absprimal}%
\end{figure}

\begin{figure}%
\centering
\includegraphics[width=0.7\columnwidth]{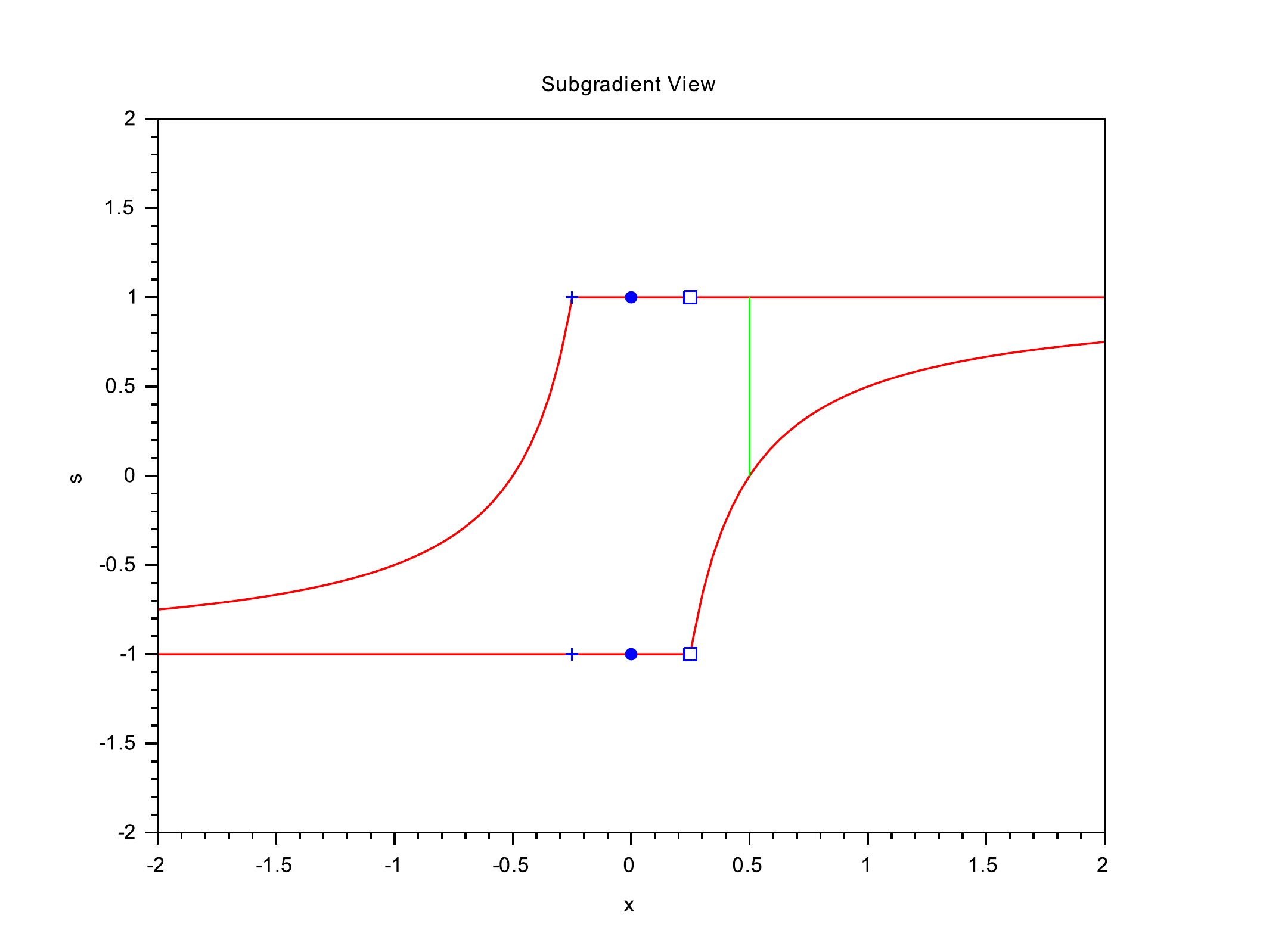}%
\caption{The $\partial_\epsilon f(\xb)$ for $f(x)=|x|$, and $\epsilon=0.5$. The green line shows $\partial_\epsilon f(\xb)$. The lower bound is the function $\ubar{g}=\inf \partial_\epsilon f$; it is composed of 2 pieces: a constant piece with value $-1$ for any $\xb\leq 0.25$, and a square root piece for $\xb > 0.25$. The cross (resp. dot, square) indicates breakpoints in the upper-bound curve (resp. PLQ function $f$, lower-bound curve).}%
\label{f:absdual}%
\end{figure}

The algorithm outline is as follow. The first step of the algorithm is to handle special cases: indicator functions of a point, linear functions, and quadratic functions, i.e. when the function $f$ is not piecewise defined. Next, the algorithm initializes the output data structure with the leftmost piece of the PLQ function. Then the main loop sweeps through tangent points $\tilde{x}$ associated with a given point $\bar{x}$ at which we compute $\ubar{g}(\bar{x})$. Finally, the rightmost piece is handled by considering all 3 possibilities: the function equals $+\infty$, the function is linear, or the function is quadratic.

\begin{algorithm}[th!]
\caption{Computing $\gph \partial_\epsilon f$; Initialization}
\label{a:gphInit}
\begin{algorithmic}[1]
	\Function{esub\_gph}{$\plqf, \epsilon$}
		\State n = size(plq,1); x = plq(:, 1);
		\If{n==1}
			\State Handle needle, linear, and quadratic functions
		\EndIf
		\Statex INITIALIZATION: handle 1\textsuperscript{st} interval
		\If{Domain left bounded}
			\State lb(1,:) = $[x(1), 3, \nan, 1, \nan]$;\Comment{$\partial_\epsilon f$ empty outside domain}
			\If{2\textsuperscript{nd} piece is quadratic}
				\State Call \_update\_nonsmooth\_xt
			\EndIf
		\ElsIf{1\textsuperscript{st} piece is linear}
			\State Compute largest index $\bar{i}$ for which $\cal L$ is tangent to $x(1)$
			\State xbm = esub\_compute\_xb
			\State lb(1,:) = $[ \text{xbm}, 3, \nan, \nan, s_1]$;\Comment{Unless special case}
			\If{$f$ nonsmooth at $x(1)$}
				\State Call \_update\_nonsmooth\_xt
			\EndIf
		\Else \Comment{1\textsuperscript{st} piece is quadratic}
			\State lb(1,:) = $[x(1), 1, 1, 1, \nan]$;
			\While{$(\xb, \yb - \epsilon)$ is below line $\cal L$}
				\State lb(k,:)=[x(ib), 1, 1, ib, \nan];
				\State k++; ib++; Update $(\xb, \yb - \epsilon)$				
			\EndWhile
			\State xbm = esub\_compute\_xb
			\State lb(1,:) = $[ \text{xbm}, 3, \nan, \nan, s_1]$;\Comment{Unless special case}
			\If{$f$ nonsmooth at $x(1)$}
				\State Call \_update\_nonsmooth\_xt
			\EndIf			
		\EndIf
		\algstore{abreak1}
\end{algorithmic}
\end{algorithm}

Algorithm~\ref{a:gphInit} shows the first part of our algorithm that focuses on the initialization. The cases of a needle function (indicator function of a single point), a linear function, and a quadratic function are handled directly. Then the first piece of $f$ is considered; it is either $\infty$ (domain is left-bounded), linear, or quadratic. For each case, the first row of the lb matrix is computed (our code use the variable lb to store the matrix $L$). In addition, when the first piece is linear or quadratic, the line $\cal L$ may be above several pieces of the function $f-\epsilon$. So we need to store all the resulting rows in the lb matrix. This is achieved by looping on the index ib (variable name corresponding to $\tilde{i}$) till the point $(\xb, \yb -\epsilon)$ is no longer below the line $\cal L$, where $\cal L$ has slope $\inf \partial f(\xt)$. Similarly, we need to store the same information in the matrix lb for the line $\cal L$ with slope $\sup \partial f(\xt)$, which is performed in the function \_update\_nonsmooth\_xt displayed in Algorithm~\ref{a:updateNonsmooth}.

\begin{algorithm}[th!]
\caption{Subroutine \_update\_nonsmooth\_xt}
\label{a:updateNonsmooth}
\begin{algorithmic}[1]
	\Function{\_update\_nonsmooth\_xt}{$x_1$, $y_1$, $s_1$}
		\State ib = ibStart; 
		\Statex Line $\cal L$ goes through $(x_1, y_1)$ with slope $s_1$
		\While{$(\xb, \yb - \epsilon)$ is below line $\cal L$}
				\State lb(k,:)=[x(ib), 2, it, ib, \nan];
				\State k++; ib++; Update $(\xb, \yb - \epsilon)$				
			\EndWhile
			\State xbm = esub\_compute\_xb \Comment{Solve appropriate quadratic equation}
			\State lb(k,:) = $[ \text{xbm}, 2, it, ib, \nan]$;\Comment{Unless special case} \label{al:brpt2}
	 \State \textbf{Return} updated lb
\EndFunction
\end{algorithmic}
\end{algorithm}

The main loop of the algorithm relies on the fact that breakpoints can only appear from the original breakpoints in the PLQ matrix or from the intersection of the line $\cal L$ with the graph of $f - \epsilon$; see Figure~\ref{f:mainloop}. We also need to carefully distinguish between linear and quadratic pieces to record the correct information. Algorithm~\ref{a:gphMain} describes the main loop.

\begin{figure}%
\centering
\includegraphics[width=.7\columnwidth]{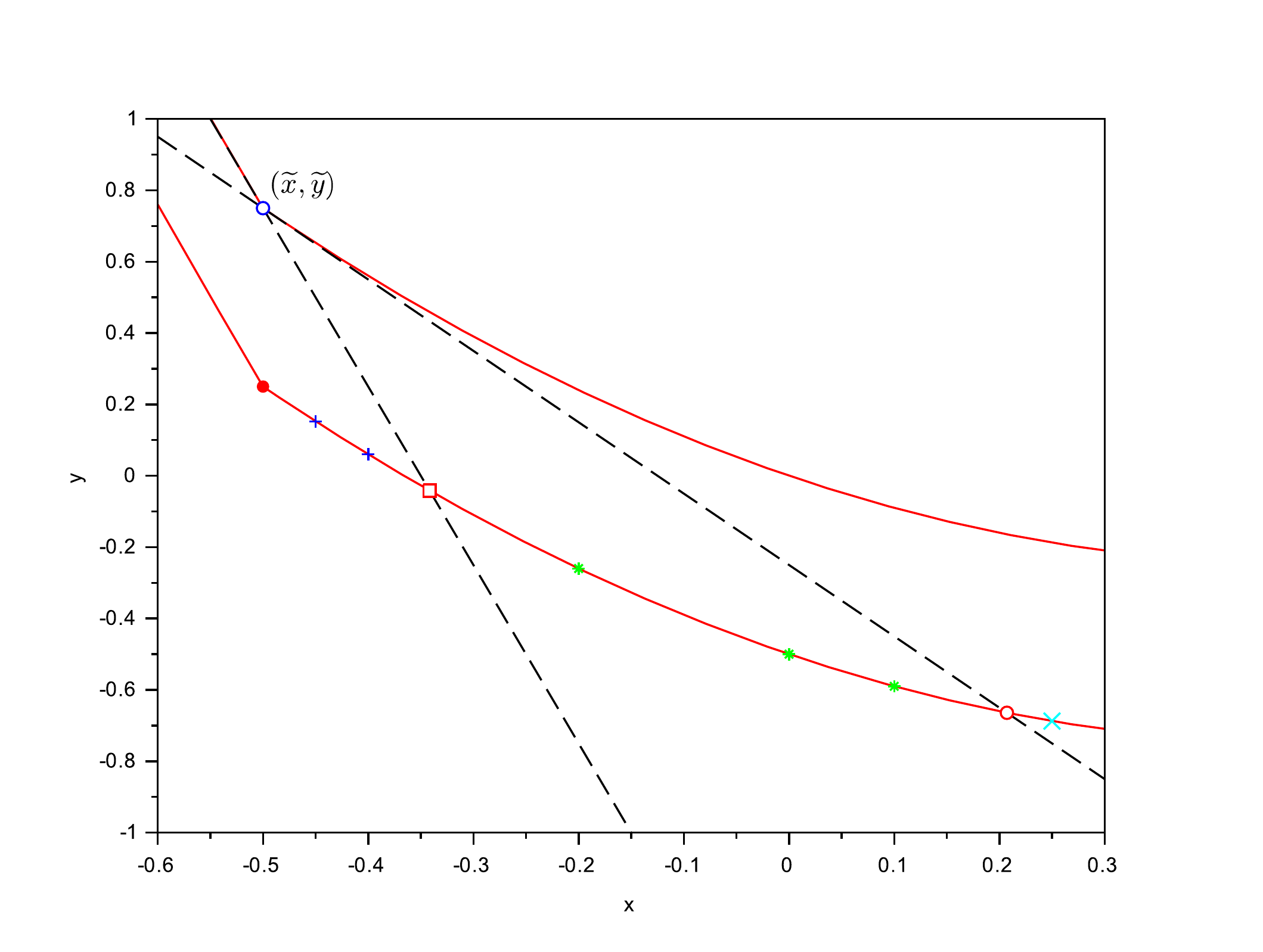}%
\caption{Main loop of the algorithm. Assume $(\xt, \yt)$ is a breakpoint between piece $i-1$ and piece $i$. The line $\cal L$ for points indicated as blue pluses is tangent to $\gph f$ on Piece $i-1$; hence these points will have type $t_k=1$. The line $\cal L$ going through points in green asterixes touches $\gph f$ at $(\xt, \yt)$ so those points have type $t_k=2$. The red square point is added to the breakpoints of $L$ storing $\ubar{g}$ in Algorithm~\ref{a:gphMain}, Line~\ref{al:brpt1} while the red circle point is added by Algorithm~\ref{a:updateNonsmooth}, Line~\ref{al:brpt2}. The point plotted as cyan multiply is part of the next iteration of the main loop.}%
\label{f:mainloop}%
\end{figure}

\begin{algorithm}
\caption{Computing $\gph \partial_\epsilon f$; Main Loop}
\label{a:gphMain}
\begin{algorithmic}[1]
	\algrestore{abreak1}
		\State it = 2
		\While{it $<$ n}
			\State Compute $\partial f(x(it))$
      \If{$f$ is linear on $[x(it-1),x(it)]$}
            \State it2use = it - 1; t = 2;
     \Else \Comment{$f$ quadratic on $[x(it-1),x(it)]$}
            \State it2use = it; t = 1;
     \EndIf			
			\While{$(\xb, \yb - \epsilon)$ is below line $\cal L$}
				\State lb(k,:)=[x(ib), t, it2use, ib, \nan];
				\State k++; ib++; Update $(\xb, \yb - \epsilon)$				
			\EndWhile
			\State xbm = esub\_compute\_xb
			\State lb(1,:) = $[ \text{xbm}, t, it2use, ib, \nan]$;\Comment{Unless special case}\label{al:brpt1}
			\If{$f$ nonsmooth at $x(it)$}
				\State Call \_update\_nonsmooth\_xt
			\EndIf			
		\EndWhile
		\algstore{abreak2}
\end{algorithmic}
\end{algorithm}

Finally, the contribution from the last piece of $f$ needs to be accounted for, which is the purpose of Algorithm~\ref{a:gphLast}. Similar to the first piece, we consider the 3 possible cases: a right-bounded domain, a linear piece, or a quadratic last piece. We then update the matrix lb accordingly.

\begin{algorithm}
\caption{Computing $\gph \partial_\epsilon f$; Process last piece}
\label{a:gphLast}
\begin{algorithmic}[1]
		\algrestore{abreak2}
    \If{domain right bounded}
				\State lb(k,:) = $[\infty, 3, \nan, n, \nan]$
		\ElsIf{$f$ is linear on rightmost part}
        \State lb(k,:) = $[\infty, 2, it-1, ib, \nan]$
    \Else \Comment{quadratic on rightmost part}
        \State lb(k,:) = $[\infty, 1, it, ib, \nan]$
    \EndIf
	 \State \textbf{Return} lb
\EndFunction
\end{algorithmic}
\end{algorithm}

\begin{proposition}
Assume $f$ is a univariate lsc convex PLQ function and $\epsilon>0$. The algorithm displayed as Algorithms~\ref{a:gphInit}, \ref{a:gphMain}, \ref{a:gphLast} compute $\gph \partial_\epsilon f$ in linear time.
\end{proposition}

\begin{proof}
All the loops in the algorithm increment either $\tilde{i}$ (variable it) or $\bar{i}$ (variable ib), hence the algorithm runs in linear time. The correctness of the algorithm follows from the definition of $\partial_\epsilon f$ and the definitions of $\xt$ and $\xb$.
\end{proof}

\section{Extension to non-PLQ functions}\label{s:extend}
All the computation in Section~\ref{s:algPt} and \ref{s:algGph} apply to convex piecewise lsc PLQ functions. However, the algorithms are easily extended to convex piecewise lsc functions for which the intersection of a line with a given piece can be computed in constant time. 

More precisely, the algorithm only requires detecting whether the piece of a convex piecewise lsc function $f$ is linear, computing $\partial f$ at any point, and running the subroutine esub\_compute\_xb. The later can be performed explicitly for convex piecewise cubic or quartic polynomials. Other functions require using a numerical method like Newton's method since we guarantee that the intersection exists on a given interval (note that there may be 2 intersection points so a little care must be exercised to select the right point).

\section{Conclusion and future work}\label{s:conc}
We proposed two algorithms. The first one is an improvement from the linear-time algorithm proposed in \cite{BAJAJ-16}. It is a dual algorithm that relies on the conjugate, the fact that a points on the graph of $f$ are in one-to-one correspondence with points on the graph of $f^*$, and a dichotomic search. The result is a logarithmic time algorithm to evaluate the $\epsilon$-subdifferential at a given point.

The second algorithm computes the full graph of $\partial_\epsilon f$. It is a line sweep algorithm that moves a point $(\xt, \yt)$ on the graph of $f$ while computing the associated point $(\xb, \yb-\epsilon)$ on the graph of $f-\epsilon$. The data structure we adopt allows the evaluation of $\partial_\epsilon f$ at a point $\xb$ in logarithmic time for a given point, or in linear-time for a given grid of points. Hence, after a linear pre-processing time, the algorithm is as efficient as the previous one for a small number of points, and more efficient (linear time vs. log-linear time) for a grid of points.

Finally, we indicated how the algorithms readily extend to convex lsc non-PLQ functions. Our data structure is particularly efficient in that regard since the algorithms are exactly the same; only the subroutine to compute the intersection of a line with a piece of $f$ has to be changed.

The algorithms for convex PLQ functions have been implemented in Scilab within the CCA numerical library~\cite{LUCET-96}. 

Future work involves implementing the algorithms for non-PLQ functions, and considering functions of 2 variables, which involves completely different data structures.

\section*{Acknowledgements}
This work was supported in part by Discovery Grants \#298145-2013 (Lucet) from NSERC, and The University of British Columbia, Okanagan campus. Part of the research was performed in the Computer-Aided Convex Analysis (CA2) laboratory funded by a Leaders Opportunity Fund (LOF, John R. Evans Leaders Fund -- Funding for research infrastructure) from the Canada Foundation for Innovation (CFI) and by a British Columbia Knowledge Development Fund (BCKDF).

This work was started at the end of Anuj Bajaj MSc research under the guidance of Dr. Warren Hare. Their preliminary efforts, ideas, and interest motivated the authors to pursue more efficient algorithms. The authors thank them for their initial contribution without which this work would not have been possible. 
	
\bibliographystyle{alpha}      %spbasic basic style, author-year citations
%\bibliography{plq_epssub}

\input{plq_epssub.bbl}
\end{document}

%% file: commands_envs.tex
\newtheorem{theorem}{Theorem}[section]

\newtheorem{example}[theorem]{Example}
\newtheorem{lemma}[theorem]{Lemma}
\newtheorem{proposition}[theorem]{Proposition}

\newtheorem{remark}[theorem]{Remark}
\newtheorem{fact}[theorem]{Fact}

%% file: conf_l.pdf_tex
%% Creator: Inkscape inkscape 0.92.0, www.inkscape.org
%% Accompanies image file 'conf_l.pdf' (pdf, eps, ps)
\begingroup%
  \makeatletter%
  \providecommand\color[2][]{%
    \errmessage{(Inkscape) Color is used for the text in Inkscape, but the package 'color.sty' is not loaded}%
    \renewcommand\color[2][]{}%
  }%
  \providecommand\transparent[1]{%
    \errmessage{(Inkscape) Transparency is used (non-zero) for the text in Inkscape, but the package 'transparent.sty' is not loaded}%
    \renewcommand\transparent[1]{}%
  }%
  \providecommand\rotatebox[2]{#2}%
  \ifx\svgwidth\undefined%
    \setlength{\unitlength}{720.49349535bp}%
    \ifx\svgscale\undefined%
      \relax%
    \else%
      \setlength{\unitlength}{\unitlength * \real{\svgscale}}%
    \fi%
  \else%
    \setlength{\unitlength}{\svgwidth}%
  \fi%
  \global\let\svgwidth\undefined%
  \global\let\svgscale\undefined%
  \makeatother%
  \begin{picture}(1,0.52808976)%
    \put(0,0){\includegraphics[width=\unitlength,page=1]{conf_l.pdf}}%
    \put(0.17523461,0.003928){\color[rgb]{0,0,0}\makebox(0,0)[lb]{\smash{$s_l$}}}%
    \put(0.29192441,0.00271624){\color[rgb]{0,0,0}\makebox(0,0)[lb]{\smash{$\ubar{s}$}}}%
    \put(0.36341009,0.003928){\color[rgb]{0,0,0}\makebox(0,0)[lb]{\smash{$s_{l+1}$}}}%
    \put(0.71927549,0.00271624){\color[rgb]{0,0,0}\makebox(0,0)[lb]{\smash{$\bar{s}$}}}%
    \put(0.04908346,0.51491519){\color[rgb]{0,0,0}\makebox(0,0)[lb]{\smash{$f^*$}}}%
    \put(0.88168104,0.3362011){\color[rgb]{0,0,0}\makebox(0,0)[lb]{\smash{$l_{\bar{x}}$}}}%
  \end{picture}%
\endgroup%